\documentclass[12pt, reqno]{amsart}

\usepackage{amssymb,amsthm,amsfonts,amsmath}

\usepackage{hyperref}
\usepackage{mathrsfs}

\newcommand{\giv}{\,|\,}

\newcommand{\ka}{\varkappa}
\newcommand{\ed}{\stackrel{d}{=}}
\newcommand{\tP}{\widetilde{P}}
\newcommand{\prob}{{\mathbb P}}

\newcommand\Nat{\mathbb N}

\newtheorem{theorem}{Theorem}

\theoremstyle{definition}

\begin{document}

\title{A Species Sampling Model with Finitely many Types }

\author{Alexander Gnedin}\thanks{Department of Mathematics, 
Utrecht University, PO Box 80 010, 3508 TA Utrecht, The Netherlands; e-mail A.V.Gnedin@uu.nl}

\maketitle

\begin{abstract}\noindent  
A  two-parameter  family of  exchangeable partitions with a
simple updating rule is introduced.
The partition is identified with a randomized version of 
a  standard symmetric Dirichlet species-sampling model with finitely many types.
A power-like distribution for the number of types is derived.

\end{abstract}

\section{Introduction} \label{Intro}
\noindent
The Ewens-Pitman two-parameter family of exchangeable partitions $\Pi^{\alpha,\theta}$
of an infinite set 
has become a central model for species sampling  
(see \cite{GHaulkP, CSP} for extensive background on exchangeability and properties of these partitions).
One most attractive feature of this model is the following explicit 
rule of succession, which we formulate as  sequential allocation of 
balls labelled $1,2,\dots$ in a series of boxes.
Start with  box $B_{1,1}$ with a single ball $1$.
At step $n$ the allocation of $n$ balls is a certain 
random partition $\Pi_n^{\alpha,\theta}$
of the set of balls $[n]:=\{1,\dots,n\}$  into some number $K_n$ of nonempty boxes
$B_{n,1},\dots,B_{n,K_n}$, which we identify with their contents,
and list the boxes  by increase of the  
minimal labels of balls\footnote{Which means that
$\min([n]\setminus (\cup_{i=1}^{j-1} B_{n,i})) \in B_{n,j}$ for $1\leq j\leq K_n$.}.
Given at step $n$ the number of occupied boxes is 
$K_n=k$,  and the occupancy counts are $\#B_{n,j}=n_j$ for $1\leq j\leq k$ (so $n_1+\dots+n_k=n$), 
the partition $\Pi_{n+1}^{\alpha,\theta}$ of $[n+1]$ 
at step $n+1$ is obtained by randomly placing ball
$n+1$ according to the rules
\begin{itemize}
\item[$({\rm O}^{\alpha,\theta})$]:  in an {\it old} box $B_{n+1,j}:=B_{n,j}\cup\{n+1\}$ with probability 
$$\omega_{n,j}^{\alpha,\theta}(k;n_1,\dots,n_k):={n_j-\alpha\over n+\theta}$$
\item[$({\rm N}^{\alpha,\theta})$]: in a {\it new} box $B_{n+1,k+1}:=\{n+1\}$ with probability 
$$\nu_n^{\alpha,\theta}(k;n_1,\dots,n_k):= {\theta+k\alpha\over n+\theta}.$$
\end{itemize}
For instance, if at step $n=6$ the  partition is  $\{1,3\},\{2,5,6\}, \{4\}$,
then ball $7$ is added to one of the old boxes $\{1,3\}$ or $\{2,5,6\}$ or $\{4\}$ with probabilities specified by $({\rm O}^{\alpha,\theta})$,
and a new box $\{7\}$ is created according to $({\rm N}^{\alpha,\theta})$.

Eventually, as all balls $1,2,\dots$ get allocated in boxes $B_j:=\cup_{n\in \Nat}B_{n,j}$, 
the collection of occupied boxes is almost surely
infinite if the parameters $(\alpha, \theta)$ are in the range 
$\{(\alpha,\theta): 0\leq\alpha< 1, \theta>-\alpha\}$. In contrast to that,
the collection of  occupied boxed has finite cardinality $\varkappa$ if
$\alpha<0$ and $-\theta/\alpha=\varkappa$;
this model,  
which is most relevant to the present study, 
has a long history going back to Fisher \cite{Fisher}.

\par Exchangeable {\it Gibbs partitions} \cite{CSP, Gibbs} extend the Ewens-Pitman
 family\footnote{Here we are only
interested in infinite exchangeable Gibbs partitions. 
Finite Gibbs partitions of $[n]$ were discussed in \cite{CSP, Berestycki}, 
but these  are typically not consistent 
as $n$ varies.}.
The first rule is preserved in the sense that, 
given ball $n+1$ is placed in one of the old boxes, it is 
placed in box $j$ with probability $\omega_{n,j}$ still proportional to $n_j-\alpha$, where 
$\alpha<1$ is a fixed 
{\it genus} of the partition. 
But the second rule allows more general functions  $\nu_n$ of $n$ and $k$,
which agree with the first rule and the exchangeability of partition.
Examples of Gibbs partitions of genus $\alpha\in (0,1)$ 
were studied in \cite{Gibbs, James, Lijoi};
from the results of these papers one can extract complicated formulas  
expressing $\nu_n$ in terms of special functions.
See \cite{Lijoi2} for a survey of related topics and applications of random partitions 
to the Bayesian nonparametric inference.

The practitioner willing to adopt a partition from the Ewens-Pitman family as a species-sampling 
model faces the dilemma:
 the total number of boxes is either a fixed finite number (Fisher's subfamily) 
or it is infinite.
For applications it is desirable to also have  
tractable exchangeable partitions of $\Nat$ with finite but random number of boxes $K$. 
The present note suggests a two-parameter family of  partitions of the latter kind,
which are Gibbs partitions obtained by suitable mixing of Fisher's $\Pi^{-1,\theta}$-partitions over 
$\theta$,
where the randomized $\theta$ (which will be re-denoted $\ka:=\theta$) 
actually coincides with $K$. 
Equivalently, the partition can be generated by sampling from a mixture of symmetric 
Dirichlet random measures with  unit weights on $\ka$ points.

\section{Construction of the partition}\label{Section2}

\noindent
A new allocation rule is as follows.
 Start with  box $B_{1,1}$ containing a single ball $1$.
At step $n$ the allocation of $n$ balls is a certain 
random partition  $\Pi_n=(B_{n,1},\dots,B_{n,K_n})$
of the set of balls $[n]$.
Given the number of boxes is 
$K_n=k$,  and the occupancy counts are $\#B_{n,j}=n_j$ for $1\leq j\leq k$, 
the partition of $[n+1]$ 
at step $n+1$ is obtained by randomly placing ball
$n+1$ 
\begin{itemize}
\item[(O)]:  in an {\it old} box $B_{n+1,j}:=B_{n,j}\cup\{n+1\}$ with probability 
$$\omega_{n,j}(k;n_1,\dots,n_k):= {(n_j+1)(n-k+\gamma) \over n^2+\gamma n+\zeta},~~j=1,\dots,k,$$
\item[(N)]: in a {\it new} box $B_{n+1,k+1}:=\{n+1\}$ with probability 
$$\nu_n(k;n_1,\dots,n_k):={k^2-\gamma k+\zeta\over n^2 +\gamma n+\zeta}.$$
\end{itemize}
To agree with the rules of probability the parameters $\gamma$ and $\zeta$ must be chosen so that
$\gamma\geq 0$  and (i) either $k^2-\gamma k+\zeta$ is (strictly) positive  for all 
$k\in {\mathbb N}$,  or (ii) the quadratic is positive for $k\in \{1,\dots,k_0-1\}$ and has a root at
$k_0$. In the case (ii) the number of occupied boxes never exceeds $k_0$.
Part (O)  is similar to the $({\rm O}^{\alpha,\theta}$)-prescription with $\alpha=-1$: given 
ball $n+1$ is placed in one of the old boxes,
it is placed in box $j$ with probability proportional to $n_j+1$.
But part (N) is radically different from $({\rm N}^{\alpha,\theta}$) in that the probability  
of creating a new box is a ratio of quadratic polynomials in $k$ and $n$.

The probability of every particular partition $B_{n,1},\dots,B_{n,K_n}$ with $K_n=k$ boxes containing
 $n_1,\dots,n_k$ balls is easily calculated as
\begin{equation}\label{EPPF}
p(n_1,\dots,n_k)= {(\gamma)_{n-k}   \prod_{i=1}^{k-1} (i^2-\gamma i+\zeta)\over \prod_{m=1}^{n-1} (m^2+\gamma m+\zeta) }
\prod_{j=1}^k n_j!\,~~~~
\end{equation}
(with $(a)_m:=a(a+1)\dots(a+m-1)$), where $(n_1,\dots,n_k)$ is an arbitrary composition of integer $n$,
that is a vector of some length $k\in\Nat$ whose components $n_j\in\Nat$ satisfy $\sum_{j=1}^k n_j=n$. 
The function $p$ is sometimes called {\it exchangeable partition probability function} (EPPF) 
\cite{CSP}.
For instance, the probability that the set of balls $[6]$ is allocated  after completing step 6  
in three boxes 
$\{1,3\},\{2,5,6\}, \{4\}$ is equal to $p(2,3,1)$.
Formula (\ref{EPPF}) is a familiar {\it Gibbs form} of exchangeable partition of genus $\alpha=-1$
(see \cite{Gibbs, CSP} and Section \ref{Gib}).
An exchangeable partition $\Pi$ of the infinite set of balls $\Nat$ is defined as the allocation of balls 
in boxes $B_j:=\cup_{n\in\Nat} B_{n,j},$
with the convention $B_j=\varnothing$ in the event $K_n<j$ for all $n$.
For $\gamma=0$ the partition $\Pi$ has only  
singleton boxes.

The formula for 
\begin{equation}\label{v_nk}
v_{n,k}:= 
 {(\gamma)_{n-k}   \prod_{i=1}^{k-1} (i^2-\gamma i+\zeta)\over \prod_{m=1}^{n-1} 
(m^2+\gamma m+\zeta) }
\end{equation}
can be  fully split in linear factors as
$$v_{n,k}=   {(\gamma)_{n-k}  (s_1+1)_{k-1}(s_2+1)_{k-1} 
\over (z_1+1)_{n-1} 
(z_2+1)_{n-1}},$$
by factoring the quadratics as
$$x^2+\gamma x+\zeta=(x+z_1)(x+z_2),~\quad   x^2-\gamma x+\zeta=(x+s_1)(x+s_2),$$
for some complex $z_1,z_2,s_1,s_2$.

Using exchangeability and applying Equation (20) from \cite{PitmanPTRF}, 
the total probability that the occupancy counts 
are $(n_1,\dots,n_k)$ equals
\begin{eqnarray*}
\prob(K_n=k, \#B_{n,1}=n_1,\dots,\#B_{n,K_n}=n_k)= \\
{n!\over \prod_{j=1}^k \{(n_j+\dots+n_k)(n_j-1)!\}}
~p(n_1,\dots,n_k)=\\
v_{n,k} n!\prod_{j=1}^k  {n_j\over n_j+\dots+n_k}
\end{eqnarray*}
Let $K_{n,r}=\#\{1\leq j\leq K_n: \#B_{n,j}=r\}$ be the number of boxes occupied by exactly $r$ out of $n$
balls.
By standard counting arguments, the last formula can be re-written as  
$$
\prob(K_{n,r}=k_r,~r=1,\dots,n)= 
v_{n,k} n!
\prod_{r=1}^n {1\over k_r!}
$$
for arbitrary integer vector of multiplicities $(k_1,\ldots,k_n)$, with $k_r\geq 0$, $\sum_{r=1}^n k_r=k$ and $\sum_{r=1}^n r k_r=n$.

\section{Mixture representation and the number of occupied boxes}
\noindent
Like for any Gibbs partition of genus $-1$,
the number of occupied boxes
$K_n$ is a sufficient statistic for 
the finite partition $\Pi_n$,
meaning that conditionally given $K_n=k$ the probability of each particular 
value of $\Pi_n$ with occupancy counts $n_1,\dots,n_k$ equals
$${\prod_{j=1}^k n_j!\over d_{n,k}},$$
where the normalization constant is a Lah number \cite{Charalambides}
\begin{equation}\label{lah}
d_{n,k}={n-1\choose k-1}{n!\over k!}.
\end{equation}
The sequence $(K_n, n=1,2,\dots)$ is a nondecreasing Markov chain with $0-1$ increments
and transition probabilities determined by the rule (N). 
The distribution of $K_n$ is  calculated as 
\begin{equation}\label{distKn}
\prob(K_n=k)= d_{n,k}v_{n,k}.
\end{equation}

By monotonicity, the limit  $K:=\lim_{n\to\infty} K_n$ 
exists almost surely,
and coincides with the number of nonempty boxes for the infinite partition $\Pi$. 
Letting $n\to\infty$ in (\ref{distKn})
and using the standard asymptotics $\Gamma(n+a)/\Gamma(n+b)\sim n^{a-b}$
we derive from (\ref{v_nk}), (\ref{lah}) 

\begin{equation}\label{distKgen}
{\mathbb P}(K=\ka)={\Gamma(z_1+1)\Gamma(z_2+1)\over \Gamma(\gamma)}   \,\,
{\prod_{i=1}^{\ka-1}(i^2-\gamma i+\zeta)\over \ka!(\ka-1)!}.
\end{equation}

The basic structural result about $\Pi$ is the following:

\begin{theorem}\label{Thm}
Partition $\Pi$ is a mixture of partitions $\Pi^{-1,\ka}$ 
over the parameter $\varkappa$, with a proper  mixing distribution 
 given by {\rm (\ref{distKgen})}.
\end{theorem}
\begin{proof}
As every other Gibbs partition of genus $-1$, partition $\Pi$ satisfies the conditioning relation
\begin{equation}\label{mix}
\Pi\,|\,\{K=\ka\}   ~~      \ed \Pi^{-1,\ka},
\end{equation}
which says that given $K$ the partition has the same distribution as some 
Fisher's partition of genus $-1$.
We only need to verify that the weights in (\ref{distKgen}) add up to the unity.

To avoid calculus, note that by the general theory \cite{Gibbs} 
$\Pi$ is a mixture of the $\Pi^{-1,\ka}\,$'s with $\ka\in {\mathbb N}$, and the trivial singleton 
partition. Because every $\Pi^{-1,\ka}$ has $\ka$ boxes, the probability
 ${\mathbb P}(K=\infty)$ is equal to the 
weight of the singleton component in the mixture.
But for the singleton partition of $[n]$
the number of boxes is equal to the number of balls, thus
it remains to check that ${\mathbb P}(K_n=n)\to 0$ as $n\to\infty$, which is easily done by inspection
 of the transition
rule (N).
\end{proof}

As $\ka\to\infty$, the masses (\ref{distKgen}) exhibit a power-like decay,
$${\mathbb P}(K=\ka)\sim\,\,{c\over \ka^{\gamma+1}} 
 \quad {\rm with ~~~~} 
c={\Gamma(z_1+1)\Gamma(z_2+1)\over\Gamma(\gamma)\Gamma(s_1+1)\Gamma(s_2+2)}
\,.$$
This explains, to an extent, the role of parameter $\gamma$. 
In particular, ${\mathbb E}K$ may be finite or infinite, depending on whether $\gamma> 1$, 
or $\gamma\leq 1$.

\section{Frequencies }

\noindent
Recall some standard facts about the
partition $\Pi^{-1,\ka}$ (see \cite{CSP}). 
This partition with $\ka$ boxes $B_1,\dots,B_\ka$ can be generated by the following steps:
\begin{itemize} 
\item[(b)] choose a value $(y_1,\dots,y_\ka)$ for the probability vector $(P_{\ka,1},\ldots,P_{\ka,\ka})$ 
uniformly distributed on the $(\ka-1)$-simplex
$\{(y_1,\dots,y_\ka): y_i>0, \sum_{i=1}^\ka y_i=1\}$,
\item[(c)] allocate balls $1,2,\ldots$ independently in 
$\ka$ boxes  with probabilities $y_1,\dots,y_\ka$ of placing a ball in each of these boxes, 
\item[(d)] arrange the boxes by increase of the smallest labels of balls.
\end{itemize}
The vector of {\it frequencies} $(\tP_{\ka,1},\dots,\tP_{\ka,\ka})$, defined through limit proportions
 \begin{equation}\label{fre}
\tP_{\ka,j}:=\lim_{n\to\infty}{\#(B_{j}\cap[n])\over n}, \quad 1\leq j\leq \ka, 
\end{equation}
has the same distribution as the {\it size-biased permutation} 
of $(P_{\ka,1},\dots,P_{\ka,\ka})$.
The frequencies have a convenient stick-breaking representation 
\begin{equation}\label{stbr}
\tP_{\ka,j}=W_j\prod_{i=1}^{j-1} (1-W_i), \quad{\rm with~~independent~~} W_i\ed{\rm beta}(2,\ka-i),
\end{equation}
where $i=1,\dots,\ka$ and ${\rm beta}(2,0)$ is a Dirac
mass at $1$. See \cite{GHaulkP} for characterizations of 
this and other Ewens-Pitman partitions 
through independence of factors in such a stick-breaking representation.

\par Now let us apply the above to the partition $\Pi$. 
The mixture representation in Theorem \ref{Thm}  implies that
$\Pi$ can be constructed by first

\begin{itemize}
\item[(a)] choosing a value $\ka$ for $K$ from distribution (\ref{distKgen}),
\end{itemize}
then following the above steps (b), (c) and (d).
The frequencies $(\tP_1,\dots,\tP_K)$ of nonempty boxes $B_1,\dots,B_K$
are obtainable from (\ref{stbr}) by mixing with weights (\ref{distKgen}).

\section{Exchangeable sequences}

\noindent
Let $(S,{\mathcal B},\mu)$ be a Polish  space with a nonatomic probability measure $\mu$.
Let $\Pi$ be the partition of $\Nat$ constructed above 
and $T_1,T_2,\dots$ be an i.i.d. sample from $(S,{\mathcal B},\mu)$, also independent of $\Pi$.
With these random objects one naturally associates an infinite exchangeable $S$-valued sequence 
$X_1,X_2,\dots$ with marginal distributions $\mu$, as follows
(see \cite{Aldous, HPitman}). 
Attach to every ball in box $B_j$ the same tag $T_j$, for $j=1,\dots, K$. 
Then define $X_1,X_2,\dots$ to be the sequence of tags of balls $1,2,\dots$.

Obviously, $K,T_1,\dots,T_K$ and $\Pi$ can be recovered from  $X_1,X_2,\ldots$. Indeed,
$T_j$ is the $j$th distinct value in the sequence $X_1,X_2,\ldots$
and  $B_j=\{n: X_n=T_j\}$ for $j=1,\dots,K$.
The same applies to finite partitions $\Pi_n$ with 
 $B_{n,j}=B_j\cap[n]$.

The {\it prediction rule} \cite{HPitman} associated with $X_1,X_2,\dots$ is the formula for conditional distribution 
$$\prob(X_{n+1}\in ds\giv X_1,\dots,X_n)=\sum_{j=1}^{K_n}  \omega_{n,j} \delta_{T_j}(ds) + \nu_n\mu(ds),$$ 
where $T_1,\dots,T_{K_n}$ are the distinct values in $X_1,\dots,X_n$, 
and $\omega_{n,j},\nu_n$ are the functions of the partition $\Pi_n$, 
as specified by the rules (O) and (N).

The random measure $F$ in de Finetti's representation of $X_1,X_2,\dots$  is a mixture  
$$F=\sum_{\ka=1}^\infty {\mathbb P}(K=\ka)\,F_\ka\,,$$ 
where
$$
F_\ka(ds)=\sum_{j=1}^\ka P_{\ka,j}\delta_{\hat{T}_j}(ds), \quad\ka\in \Nat$$
are
Dirichlet$(\underbrace{1,\dots,1}_\ka)$ random measures on $(S,{\mathcal B},\mu)$, that is the vector
$(P_{\ka,1},\dots,P_{\ka,\ka})$ is uniformly distributed on the $(\ka-1)$-simplex and is independent of 
$(\hat{T}_1,\hat{T_2},\dots)$,
and the random variables $\hat{T}_j$'s are i.i.d.$(\mu)$.

\section{The case $\zeta=0$}
\noindent
We focus now on the case $\zeta=0$. Then $\gamma\in [0,1]$
is the admissible range, but we shall  
exclude the trivial edge cases $\gamma=0$, respectively, $\gamma=1$ of the singleton and 
single-box partitions.

Formula (\ref{v_nk}) simplifies as 
$$v_{n,k}= 
 {(k-1)! (1-\gamma)_{k-1} (\gamma)_{n-k}\over (n-1)!(1+\gamma)_{n-1}},$$
and there is a further  obvious cancellation of some factors.
Furthermore, (\ref{distKgen}) specializes as
\begin{equation}\label{distK}
\prob(K=\ka)={\gamma(1-\gamma)_{\ka-1}\over \ka!},~~~\ka=1,2,\dots
\end{equation}
which is a distribution familiar 
from the discrete renewal theory 
(a summary is found in \cite{PitmanBernoulli}, p. 85). 
The distribution
has also appeared in connection 
with $\Pi^{\alpha,\theta}$ ($0<\alpha<1$)
partitions and other occupancy problems \cite{RegCS, GHansenP, GPYI, PitmanBernoulli}.

Thinking  of (\ref{distK}) as a prior distribution for $K$, 
the posterior distribution is found 
from (\ref{distKn}), (\ref{distK}) and the distribution of the number of occupied 
boxes for the $\Pi^{-1,k}$ partition (instance of Equation (3.11) in \cite{CSP}):
\begin{equation}\label{terminal}
\prob(K=\varkappa\giv K_n=k)=
{(n-1)!\over (k-1)!(\varkappa+n-1)!}\prod_{i=1}^{k-1}(\varkappa-i)
\prod_{j=1}^k(\gamma+n-j)\prod_{l=k}^{\varkappa-1} (l-\gamma), ~~~~
\end{equation}
for $1\leq k\leq n, ~~~\varkappa\geq k$. Note that the conditioning here
can be replaced by conditioning on an arbitrary value of the partition $\Pi_n$ with $k$ boxes.


The frequency $\tP_1$ of box $B_1$ has distribution
$$\prob(\tP_1\in dy)= \sum_{\ka=1}^\infty  {\gamma(1-\gamma)_{\ka-1}\over\ka!}\,\,  
\prob(\tP_{\ka,1}\in dy)=    \gamma\delta_1(dy)+ (1-\gamma)\gamma y^{\gamma-1}dy, \quad y\in (0,1],$$  
which is a mixture of Dirac mass at 1 and beta$(\gamma,1)$ density. Interestingly,
distributions of this kind have appeared in connection with other partition-valued processes
\cite{RegCS, GYak}.
The distribution is useful to compute expected values of symmetric statistics of the frequencies 
of the kind
$\sum_{j=1}^K f(\tP_j)$ \cite{CSP},
for example
$${\mathbb E}\left(\sum_{j=1}^K \tP_j^{n}\right)= {\mathbb E}\left(\tP_1^{n-1}\right)={n\gamma\over n+\gamma-1},$$ 
which agrees with the $\prob(K_n=1)$ instance of (\ref{distKn}).

\section{Restricted  exchangeability}

\newcommand{\bb}{{\bf b}}
\noindent
It is of interest to explore
a more general situation when the  process starts with some initial allocation   
of a few balls in boxes. This can be thought of as prior information of the observer about the
existing species. 
For simplicity we shall only consider the case $\zeta=0$.

Fix $m\geq 1$ and a partition $\bb=(b_1,\dots,b_k)$ of $[m]$ with $k$ positive box-sizes $\#b_j=m_j,~j=1,\dots,k$.
Let $\prob_\bb$ be the law of the infinite partition $\Pi$  
constructed by the rules (O) and (N) starting with the initial allocation of balls $\Pi_m=\bb$. In particular, 
$\prob=\prob_{\{1\}}$.
Note that $\prob_\bb$ is well defined for any value of the parameter in the range
$$-(m-k)<\gamma<k,$$
and for  $\gamma\in (0,1)$ the measure
$\Pi_{\{1\}}$, conditioned on 
$\{\Pi_m=\bb\}$, coincides with  $\prob_\bb$.
Explicitly, under $\prob_\bb$ every value of  $\Pi_n=(B_{n,1},\dots,B_{n,K_n})$ with
$$K_n=\varkappa\geq k,~~\#B_j=n_j, ~j=1,\dots,\ka; \quad n_i\geq m_i,~i=1,\dots,k$$ 
has probability
$$ p_\bb(n_1,\dots,n_\ka):= {p(n_1,\dots,n_{\ka})\over p(m_1,\dots,m_k)}\,,$$
where $p$ is given by (\ref{EPPF}). Formula (\ref{distK}) for the terminal distribution of the number of boxes
is still valid for the extended range of $\gamma$.

\par Observing that $p_\bb$ is symmetric in the arguments $n_j$ for $k\leq j\leq \ka$, it follows that $\prob_\bb$ 
is invariant under
permutations of the set $\Nat\setminus[m]$. 
On the other hand, 
for every permutation  $\sigma:[m]\to[m]$ we have
$\prob_{\sigma b}=\sigma \prob_b$.  
Moreover, the restriction of $\Pi$ on $\Nat\setminus[m]$
under $\prob_\bb$ has the same law as under $\prob_{\sigma \bb}$, that is the restriction depends on $\bb$ only through
 $(m_1,\dots,m_k)$.

{\bf Examples} Suppose  $\gamma=1$. Then 
$\prob_{\{1\}}(K=1)=1$ which corresponds to the trivial one-block partition, but 
$\prob_{\{1\},\{2\}}(K=\ka)={2\over \ka(\ka+1)}$ for $\ka\geq 2$.

Suppose $\gamma=0$.
Then $\Pi$  under $\prob_{\{1\}}$ is the trivial singleton partition, but under $\prob_{\{1,2\}}$ we have
$\prob_{\{1,2\}}(K=\ka)={1\over \ka(\ka+1)}$ for $\ka\geq 1$.

\section{General Gibbs partitions and the new family}\label{Gib}
\noindent
Both the Ewens-Pitman family and the partitions introduced in this note can be constructed in  
a unified way, using simple algebraic identities.
Recall from \cite{Gibbs, CSP} that the  Gibbs form for EPPF 
  $p$ 
of genus $\alpha\in (-\infty,1)$ is \footnote{We omit here the case 
$\alpha=-\infty$.}
$$p(n_1,\dots,n_k)=v_{n,k}\prod_{j=1}^k (1-\alpha)_{n_j-1},$$
where the triangular array $(v_{n,k})$ 
is  nonnegative and
 satisfies the recursion 
\begin{equation}\label{vrec}
v_{n,k}=(n-k\alpha)v_{n+1,k}+v_{n+1,k+1},\qquad 1\leq k\leq n
\end{equation}
with normalization ~$v_{1,1}=1$.  
The recursion goes backwards, from $n+1$ to $n$, thus it cannot be `solved' in a unique way
rather has a convex set of solutions, each corresponding to distribution of some exchangeable partition.

For Gibbs partition 
the number of occupied boxes $(K_n, n=1,2,\dots)$ 
is a  nondecreasing Markov chain, viewed conveniently as a bivariate space-time walk $(n,K_n)$,
which has backward transition probabilities 
depending on $\alpha$ but not on $(v_{n,k})$.
The backward transition probabilities are determined from the conditioning relation: 
given $(n,K_n)=(n,k)$, the probability of each admissible path from 
$(1,1)$ to $(n,k)$ is proportional to the product of {\it weights} along the path,
where 
the weight of transition $(n,k)\to(n+1,k)$ is $n-k\alpha$, and that of $(n,k)\to(n+1,k+1)$ is 
$1$. The normalizing total sum $d_{n,k}(\alpha)$ of such products over the paths from $(1,1)$
 to $(n,k)$ is known as a  generalized Stirling number
\cite{Charalambides}.
Each particular solution to (\ref{vrec}) determines the  law of $(K_n)$ 
via the marginal distributions
${\mathbb P}(K_n=k)=v_{n,k}d_{n,k}(\alpha)$.

Large-$n$ properties of Gibbs partitions depend on $\alpha$. 
In particular, there exists an almost-sure limit
$K=\lim_{n\to\infty}K_n/c_n(\alpha)$, where 
 $c_n(\alpha)=n^\alpha$ for $\alpha\in (0,1)$, 
$c_n(0)=\log n$ and $c_n(\alpha)\equiv 1$ for $\alpha<0$.
The law of $K$ is characteristic for partition  of given genus.
That is to say, a generic Gibbs partition is a unique mixture over $\varkappa$ 
of {\it extreme}  partitions for which $K=\varkappa$ a.s. 
Note that $K$ has continuous range  for $\alpha\in [0,1)$, 
and discrete for $\alpha<0$.
For $\alpha< 0$ the extremes are Fisher's partitions $\Pi^{\alpha,-\alpha\varkappa}$.
For $\alpha=0$ the extremes are Ewens' partitions $\Pi^{0,\varkappa}$ (with $\ka\in [0,\infty]$).
For $\alpha\in (0,1)$ Ewens-Pitman partitions are not extreme,
rather the extremes are obtainable by conditioning any
$\Pi^{\alpha,\theta}$ on $K=\varkappa$;
the $v_{n,k}$'s  for these extreme partitions 
were identified in \cite{James} in terms of the generalized
hypergeometric functions.

Following \cite{Boundary}, where recursions akin to  (\ref{vrec}) were treated,
one can seek for special solutions of the form 
\begin{equation}\label{vspecial}
v_{n,k}={\prod_{i=1}^{n-k} f(i)\prod_{j=1}^{k-1} g(j) \over\prod_{m=1}^{n-1}
 h(m)},
\end{equation}
where $f,g,h:{\mathbb N}\to{\mathbb R}$ satisfy the identity
\begin{equation}\label{ident} 
(n-\alpha k)f(n-k)+g(k)=h(n), \quad 1\leq k\leq n,~n\in {\mathbb N}.
\end{equation}
Moreover,
 $f, h$ must be  (strictly) positive on $\Nat$, 
while $g$ may be either positive on $\Nat$ or positive on some integer interval 
$\{1,\dots,k_0-1\}$ with 
$g(k_0)=0$.
Each such triple defines a Gibbs partition with the `new boxes' updating rule 
of the form  
$$\nu_n(k; n_1,\dots,n_k)={\mathbb P}(K_{n+1}=k+1\,|\,K_n=k)={g(k)\over h(n)}$$
(where $n=n_1+\dots + n_k$),  complemented by the associated
version of the (O)-rule (as in Sections \ref{Intro} and \ref{Section2}).

\par Now we can review two instances of (\ref{vspecial}):
\begin{itemize}
\item

Exploiting the identity  $n-\alpha k+\alpha k+\theta= n+\theta$
we may 
choose $f(n)\equiv 1, g(n)=\alpha n+\theta$ and $h(n)=n+\theta$.
This yields 
the Ewens-Pitman partitions with the
succession rule $({\rm N}^{\alpha,\theta})$. 
Note that the admissible range for $\alpha,\theta$ is determined straightforwardly from the positivity.

\item The identity 
$$(n-k+\gamma)(n+k)+k^2-\gamma k+\zeta=n^2+\gamma n +\zeta$$
is of the kind (\ref{ident}) with $\alpha=-1$.  
We choose $f(n)=n+\gamma, g(n)=n^2-\gamma n+\theta$ and $h(n)= n^2+\gamma n+\theta$
to arrive at the partitions itroduced in this paper.
\end{itemize}
It is natural to wonder if there are any other Gibbs partitions of the form  (\ref{vspecial}).

The ansatz (\ref{vspecial}) is sometimes useful to deal with recursions like (\ref{vrec}) 
with other weights depending in a simple way on $n$ and $k$ \cite{Boundary}. 
For instance, if both weights equal 1, then
each solution defines  distribution of an exchangeable $0-1$ sequence (see \cite{Aldous}),
for which the Markov chain $(K_n)$ counts the number of $1$'s among the first $n$ bits.
An instructive exercise is to construct by this method of specifying the triple
$f,g,h$ 
two distinguished families of the exchangeable processes -- 
the homogeneous Bernoulli processes and  P{\'o}lya's urns with two colors.  


\noindent
{\bf Acknowledgement} The author is indebted to J. Pitman,
an associated editor and a referee
for their stimulating questions and comments.

\end {document}